\documentclass{amsart}

\title[]
      {Quantitative Recurrence for Generic Homeomorphisms}

\author{Andre Junqueira}
\address{UNIVERSIDADE FEDERAL DE VI\c{C}OSA}
\email{andre.junqueira@ufv.br}

\thanks{Partially supported by CAPES}

\newtheorem{maintheorem}{Theorem}

\newtheorem{theorem}{Theorem}
\newtheorem{corollary}[theorem]{Corollary}
\newtheorem{claim}{Claim}

\newtheorem{lemma}[theorem]{Lemma}

\theoremstyle{definition}
\newtheorem{definition}[theorem]{Definition}
\newtheorem{remark}[theorem]{Remark}

\newcommand{\N}{\mathbb{N}}
\newcommand{\R}{\mathbb{R}}

\newcommand{\SC}{{\mathcal C}}

\newcommand{\SG}{{\mathcal G}}
\newcommand{\SH}{{\mathcal H}}

\newcommand{\SM}{{\mathcal M}}

\newcommand{\SR}{{\mathcal R}}

\newcommand{\SW}{{\mathcal W}}

\begin{document}

\begin{abstract}
In this article we study quantitative recurrence for generic homeomorphisms on euclidian spaces and compact manifolds. As an application we show that the decay of correlations of generic homeomorphisms is slow.
\end{abstract}

\date{\today}

\maketitle

\section{Introduction}
The Poincar\'{e} Recurrence Theorem states that, in a dynamical system preserving a probability measure on a polish metric space, almost every orbit returns as
closely as you wish to the initial point. To be more precise, this means that:
$$\liminf_{n\to +\infty}d(T^{n}(x),x)=0.$$
Now, if a dynamical system preserve a probability measure on a polish metric space and is ergodic then almost every orbit hit as closely as you wish on a fixed point $y$ on the support of the measure. To be more precise, we have that:
$$\liminf_{n\to +\infty}d(T^{n}(x),y)=0.$$
In \cite{Bo} Boshernitzam proved that in very general conditions
$$\liminf_{n\to\infty}n^{\beta}d(T^{n}(x),x)<+\infty$$
for almost every point with respect to an invariant probability measure and for some $\beta>0$. This result means that the speed of recurrence is not too slow with respect the sequence $n^{\beta}$. So, a natural question is if the recurrence can be arbitrarily fast and the answer is yes as we will see in theorem \ref{teoprincipalrt}.
On the other hand, in \cite{Bon} it is proved that in very general conditions
$$\liminf_{n\to\infty}n^{\alpha}d(T^{n}(x),y)=+\infty$$
for almost every point with respect to an invariant probability measure and for some $\alpha>0$. This means that the speed of hitting is not too fast with respect the sequence $n^{\alpha}$. So, a natural question is if the hitting can be arbitrarily slow and the answer is yes as we will see in theorem \ref{teoprincipalht}. We also investigate the same question but in a non-compact setting as we will see in theorem \ref{teoprincipalrtinfinite}.
\section{Preliminaries}
\subsection{Compact manifolds}
Let $M$ be a compact connected Riemannian manifold with no boundary. The measures that we will work are given by the next definition:

\begin{definition}
\label{OU}
We say that a borel probability measure on $M$ is an \emph{OU}(Oxtoby-Ulam) \emph{measure} if:
\begin{enumerate}
\item $\mu$ is \emph{nonatomic}. This means that it is zero on singleton sets.
\item $\mu$ is \emph{locally positive}. This means that it is positive on every nonempty open set.
\end{enumerate}
\end{definition}
If $\mu$ is a probability measure on $M$ and $f:M\to M$ is a measurable map then we say that $f$ preserve $\mu$ if $\mu(f^{-1}(A))=A$ for all borel sets $A$. Now, let us define the main space of this work.

\begin{definition}
\label{homeopreserving}
If $\mu$ is an OU probability measure on $M$ then we define:
$$\SM[M,\mu]:=\{f:M\to M : f\,\,is\,\,a\,\,homeomorphism\,\,preserving\,\,\mu\}$$
with the \emph{uniform topology} given by the metric:
$$||f-g||:=ess\,sup_{x\in M}d(f(x),g(x))$$   
\end{definition}

The space $\SM[M,\mu]$ is not complete with this metric but there exists an equivalent metric such that this space is complete with this new metric. Then we can apply the Baire Theorem for $\SM[M,\mu]$. See \cite{BF} for more details.
A set which is the countable intersection of open sets is called a $G_{\delta}$ set. We shall call a subset $\SR\subset \SM[M,\mu]$ \emph{generic} if it contains a dense $G_{\delta}$ set. It follows from the Baire Theorem that a generic subset is dense.
Now, let us remember the classical Poincar\'{e} recurrence theorem.

\begin{theorem}[Poincar\'{e}]
Let us suppose that $(X,d)$ is a complete and separable metric space and $T:X\to X$ a measurable map preserving a borel probability measure $\mu$ on $X$. Then:
\begin{enumerate}
\item $\liminf_{n\to +\infty}d(T^{n}(x),x)=0$ for $\mu\,\,a.e.\,\,x\in X$.
\item If $y\in supp(\mu)$ and $\mu$ is ergodic then $\liminf_{n\to +\infty}d(T^{n}(x),y)=0$ for $\mu\,\,a.e.\,\,x\in X$.
\end{enumerate}
\end{theorem}

So, an important problem is to study the speed of those limits. 

\subsection{Euclidian Spaces}

Let us denote by $\R^{n}$ the $n-dimensional$ Euclidian Space and $\lambda$ the lebesgue measure.

\begin{definition}
Let us define the following space:
$$\SM[\R^{n},\lambda]:=\{f:\R^{n}\to \R^{n}: f\,\,is\,\,a\,\,homeomorphism\,\,preserving\,\,\lambda\}$$
\end{definition}
In this space we put the compact-open topology, where the basic open sets are given by:
$$\SC(f,K,\delta):=\{g\in \SM[\R^{n},\lambda]: |g(x)-f(x)|<\delta\,\,for\,\,\lambda\,a.e.\,\,x\in K\}$$
where $K\subset \R^{n}$ is a compact set and $\delta>0$. This space can be metrized as follows: Let $K_{i}$ be a sequence of compact sets whose union is $X$. Then the compact-open topology is induced by the complete metric
$$U(f,g)=\sum_{i=1}^{+\infty}\frac{u_{K_{i}}(f,g)}{1+u_{K_{i}}(f,g)},$$
where
$$u_{S}(f,g)=\max\left(\max_{x\in S}|f(x)-g(x)|,\max_{x\in S}|f^{-1}(x)-g^{-1}(x)|\right).$$

\begin{remark}
The metric space $\SM[\R^{n},\lambda]$ is complete and then we can apply the Baire Theorem. In particular we can define generic sets are before.
\end{remark}
The Poincar\'{e} recurrence theorem also holds in this setting but we need an additional hypothesis.

\begin{definition}
We say that a measurable map $T$ on $(\R^{n},\lambda)$ is conservative if always that $E$ is a borel set such that $\{T^{-n}(E)\}_{n\geq 0}$ are disjoint then we have that $\lambda(E)=0$
\end{definition}

Now, let us remember the Poincar\'{e} recurrence theorem in this setting.

\begin{theorem}
Let us suppose that $T$ is a measurable map conservative and invariant with respect to lebesgue measure on $\R^{n}$. Then:
$$\liminf_{n\to +\infty}d(T^{n}(x),x)=0$$ 
for $\lambda\,\,a.e.\,\,x\in \R^{n}$.
\end{theorem}

It is well known that every ergodic map with respect to a non-atomic measure is conservative(see \cite{Aa}) and that the ergodicity is a generic property in $(\R^{n},\lambda)$(see \cite{BF}) and then the Poincar\'{e} recurrence theorem holds in a generic set. Then it is a natural question to study the speed of recurrence in this setting.

\section{Main statements}

%From now on $M$ will denote a compact connected Riemannian manifold with no boundary and $\mu$ an OU measure on $M$.
%It follows from theorems \ref{DGRT} and \ref{DGHT} that the recurrence time is not very slow and the hitting time is not very fast. The main Theorem of this work says that in the setting of homeomorphisms measure preserving the recurrence time and hitting time are not good to measure the recurrence because the recurrence returning to the same point can be arbitrarily fast and the recurrence in a different point can be arbitrarily slow. To be more precise we have:

\begin{maintheorem}
\label{teoprincipalrt} 
Let be $(r_{n})\longrightarrow +\infty$ in $\R^{+}$, $(Y,d)$ a metric space and $f:M\to Y$ a continuous map. Then there exists a generic set $\SR\subset \SM[M,\mu]$ such that if $T\in \SR$ then:
$$\liminf_{n\to +\infty} r_{n}d(f(T^{n}(x)),f(x))<+\infty\,\,\,for\,\,\mu\,\,a.e.\,\,x\in M.$$
\end{maintheorem}

\begin{maintheorem}
\label{teoprincipalht} 
Let be $(r_{n})\longrightarrow +\infty$ in $\R^{+}$, $(Y,d)$ a metric space, $y\in M$ and $f:M\to Y$ a continuous map such that $\mu(f^{-1}(\{f(y)\})=0$. Then there exists a generic set $\SG\subset \SM[M,\mu]$ such that if $T\in \SG$ then:
$$\liminf_{n\to +\infty} r_{n}d(f(T^{n}(x)),f(y))=+\infty\,\,\,for\,\,\mu\,\,a.e.\,\,x\in M.$$
\end{maintheorem}

\begin{maintheorem}
\label{teoprincipalrtinfinite} 
Let be $(r_{n})\longrightarrow +\infty$ in $\R^{+}$. Then there exists a generic set  $\SH\subset \SM[\R^{n},\lambda]$ such that if $T\in \SH$ then:
$$\liminf_{n\to +\infty} r_{n}d(T^{n}(x),x)<+\infty\,\,\,for\,\,\lambda\,\,a.e.\,\,x\in \R^{n}.$$
\end{maintheorem}

To finish we have a consequence on the decay of correlations. Before let us recall this notion.

\begin{definition}
\label{superpolynomial}
Let $\phi$,$\psi$:$M\to \R$ Lipschitz functions on $M$. If $(M,T,\mu)$ is a measure preserving system then we say that $T$ has a superpolynomial decay of correlations if:
$$\left|\int_{X}(\phi\circ T^{n})\psi d\mu - \int_{X}\phi d\mu \int_{X}\psi d\mu\right| \leq ||\phi|| ||\psi|| \theta_{n}$$
where $\lim_{n\to +\infty}\theta_{n}n^{p}=0$ for all $p>0$ and $||.||$ is the Lipschitz norm. 
\end{definition}

Now, let us recall the definition of local dimension with respect to a measure.

\begin{definition}
\label{localdimension}
Let $\mu$ be a borel probability measure on $M$. Then we define the local dimension of $\mu$ in $x\in M$ as:
$${d}_{\mu}(x):=\lim_{r\to 0}\frac{\log \mu(B(x,r))}{\log r}$$
if this limit there exists.
\end{definition}

Then we have the following:

\begin{corollary}
\label{genericdecay}
Let us suppose that there exists $y\in M$ such that $d_{\mu}(y)>0$. Then there exists a generic set $\SG_{1}\subset \SM[M,\mu]$ such that if $T\in\SG_{1}$ then the decay of correlations of $T$ with respect to $\mu$ is not super-polynomial.
\end{corollary}

\section{Proofs}

\begin{proof}(Theorem \ref{teoprincipalrt})
Let us begin with the following lemma:

\begin{lemma}
\label{Alpern}
Let be $\mu$ an O.U. measure on $M$. Then, for every $\epsilon>0$ the set of homeomorphisms $T\in \SM(M,\mu)$ such that $\mu(Per(T))>1-\epsilon$ is dense in $\SM[M,\mu]$, where $Per(T)$ denotes the set of periodic points of $T$.
\end{lemma}

\begin{proof}
Let be given $\epsilon>0$ and $\delta>0$. Using Vitali's covering lemma we can find a finite collection of disjoint balls $\{U_{i}\}_{i=1}^{N}$ such that $Diam(U_{i})<\delta$ for each $i$ and
$$\mu(\bigcup_{i=1}^{N}U_{i})>1-\frac{\epsilon}{2}.$$
Now, we take balls $V_{i}\subset U_{i}$ such that:
$$\mu(\bigcup_{i=1}^{N}V_{i})>1-\epsilon.$$
Now, let be $T\in \SM[X,\mu]$, $V\subset U$ balls in $X$ such that $Diam(U)<\delta$. Then it follows from corollary 9 in \cite{DF} that there exists $g\in\SM[X,\mu]$ and $p\in\N$ such that $g^{p}$ is the identity on $V$, $g=T$ on $U^{c}$ and $||g-T||\leq Diam(V)\leq \delta$.
So, there exists $g_{1}\in\SM[X,\mu]$ and $p_{1}\in\N$ such that $g_{1}^{p_{1}}$ is the identity on $V_{1}$, $g_{1}=T$ on $U_{1}^{c}$ and $||g_{1}-T||\leq Diam(V_{1})\leq \delta$. If we repeat the argument for $g_{1}$ in place of $T$ and so on we obtain $g_{2},...,g_{N}$ in $\SM[X,\mu]$ and if we define $g:=g_{N}$ then we have that $g\in\SM[X,\mu]$, $g^{p_{i}}$ is the identity on $V_{i}$ for each $i$, $g=T$ on $(U_{1}\cup ... \cup U_{N})^{c}$ and $||T-g||<\delta.$ As we have that $\mu(Per(g))>1-\epsilon$ then the proof is over.  

\end{proof}

Now, let us come back to the proof of theorem \ref{teoprincipalrt}. Given $T\in \SM[M,\mu]$, we define:
$$\SR(T):=\{x\in M : \liminf_{n\to +\infty} r_{n}d(f(T^{n}(x)),f(x))<+\infty\}.$$
Now, if we fix $T\in \SM[M,\mu]$, $k>0$ and $n\in \N$ then we define:
$$X_{n,k}^{f}(T)=\{x\in M : r_{n}d(f(T^{n}(x)),f(x))<k\}$$
Then, we have the following claim:

\begin{claim}
\label{igualdadedeconjunto}
$$\SR(T)=\bigcup_{k>0}\bigcap_{m\geq 1}\bigcup_{n\geq m}X_{n,k}^{f}(T)$$
\end{claim}

In fact, if $x\in \SR(T)$ then $\liminf_{n\to +\infty} r_{n}d(f(T^{n}(x)),f(x))=c\,\,\,\in\,\R$. If we take $k>c$ then there exists a sequence $(n_{j})$ in $\N$ such that $r_{n_{j}}d(f(T^{n_{j}}(x)),f(x))<k$ for all $j\in\N$ and this shows that $x\in \bigcup_{k>0}\bigcap_{m\geq 1}\bigcup_{n\geq m}X_{n,k}^{f}(T)$.
On the other hand if $x\in \bigcup_{k>0}\bigcap_{m\geq 1}\bigcup_{n\geq m}X_{n,k}^{f}(T)$ then there exists $k>0$ and a sequence $n_{j}$ in $\N$ such that $r_{n_{j}}d(f(T^{n_{j}}(x)),f(x))<k$ for all $j\in \N$. Then we have that $\liminf_{n\to +\infty} r_{n}d(f(T^{n}(x)),f(x))<k<+\infty$. Then $x\in\SR(T)$ and this proves the claim.   

As a consequence of the claim \ref{igualdadedeconjunto} we get the following claim:

\begin{claim}
\label{medidadeconjunto}
$\mu(\SR(T))>1-\epsilon$ if and only if there exists $k>0$ such that for all $m\in \N$ , there exists a positive integer $l>m$ such that $\mu(\bigcup_{n=m}^{l}X_{n,k}^{f}(T))>1-\epsilon$.
\end{claim}

In fact, let us suppose that $\mu(\SR(T))>1-\epsilon$. Then, we have that
$$\mu(\SR(T))=\lim_{k\to +\infty}\mu(\bigcap_{m\geq 1}\bigcup_{n\geq m}X_{n,k}^{f}(T))>1-\epsilon$$
and then there exists $k>0$ such that $\mu(\bigcap_{m\geq 1}\bigcup_{n\geq m}X_{n,k}^{f}(T))>1-\epsilon$. So, there exists $k>0$ such that for all $m\in\N$  $\mu(\bigcup_{n\geq m}X_{n,k}^{f}(T))>1-\epsilon$. It follows that there exists $l>m$ such that $\mu(\bigcup_{n=m}^{l}X_{n,k}^{f}(T))>1-\epsilon$.
On the other hand, let us suppose that there exists $k>0$ such that for all $m\in \N$, there exists a positive integer $l>m$ such that $\mu(\bigcup_{n=m}^{l}X_{n,k}^{f}(T))>1-\epsilon$. This implies that $\mu(\bigcup_{n\geq m}X_{n,k}^{f}(T))>1-\epsilon$ and letting $m\to \infty$ we have $\mu(\bigcap_{m\geq 1}\bigcup_{n\geq m}X_{n,k}^{f}(T))>1-\epsilon$. Now, it is clear that $\mu(\SR(T))>1-\epsilon$ and this finish the claim \ref{medidadeconjunto}.
\\ \\
Now, let us define for each $\epsilon>0$ and $l>m$ in $\N$ the set:
$$\SR_{\epsilon}^{m,l}:=\{T\in \SM[M,\mu]:there\,\,\,exists\,\,\,k>0\,\,\, such\,\,that\,\, \mu(\bigcup_{n=m}^{l}X_{n,k}^{f}(T))>1-\epsilon\}$$
Then we have:

\begin{claim}
\label{open}
$\SR_{\epsilon}^{m,l}$ is open in the uniform topology.
\end{claim}
Let us fix $n,k\in\N$ and $T\in\SM[M,\mu]$. If $S\in\SM[M,\mu]$ then by triangle inequality we have that $r_{n}d(f(S^{n}(x)),f(x))\leq r_{n}d(f(S^{n}(x)),f(T^{n}(x)))+r_{n}d(f(T^{n}(x)),f(x))$. So, using the continuity of $f$ we have that if $S$ is sufficiently close to $T$ in the uniform topology (see definition \ref{homeopreserving}) then $X_{n,k}^{f}(T)\subset X_{n,k}^{f}(S)$. Using the same argument, if we fix $k>0$, $l>m$ and $T\in\SM[M,\mu]$ then $\bigcup_{n=m}^{l}X_{n,k}^{f}(T)\subset\bigcup_{n=m}^{l}X_{n,k}^{f}(T)$ for every $S\in\SM[M,\mu]$ sufficiently close to $T$. This shows that $\SR_{\epsilon}^{m,l}$ is open and finish the proof of claim \ref{open}.
\\ \\
Now, let us define for each $\epsilon>0$ the set
$$\SR^{\epsilon}:=\{T\in \SM[M,\mu] : \mu(\SR(T))>1-\epsilon \}.$$
It follows from claim \ref{igualdadedeconjunto} that $\SR(T)=\bigcup_{k>0}\bigcap_{m\geq 1}\bigcup_{n\geq m}X_{n,k}^{f}(T)$. Now, using the claim \ref{medidadeconjunto} we get that:
$$\SR^{\epsilon}=\bigcap_{m\in\N}\bigcup_{l>m}\SR_{\epsilon}^{m,l}$$
which shows that $\SR^{\epsilon}$ is a $G_{\delta}$ set for each $\epsilon>0$.
If we show that $\SR^{\epsilon}$ is dense in $\SM[M,\mu]$ for each $\epsilon>0$ then 
$\SR:=\bigcap_{n\in\N}\SR^{1/n}$ will be the generic of Theorem \ref{teoprincipalrt}.
Now, note that $Per(T)\subset \SR(T)$ and then we have that
$\{T\in \SM[M,\mu] : \mu(Per(T))>1-\epsilon\}\subset \SR^{\epsilon}$. So it follows from lemma \ref{Alpern} that $\SR^{\epsilon}$ is dense and this proves the Theorem \ref{teoprincipalrt}.
\end{proof}

\begin{proof}(Theorem \ref{teoprincipalht})
Given $T\in \SM[M,\mu]$ and $p\in\N$ we define:
$$\SW_{p}^{f}(y,\{r_{n}\},T):=\{x\in M : d(f(T^{n}(x)),f(y))<\frac{p}{r_{n}}\,\,\,for\,\,infinitely\,\,many\,\,n\}$$
and note that:
$$\SW_{p}^{f}(y,\{r_{n}\},T)=\bigcap_{m\geq 1}\bigcup_{n\geq m}T^{-n}(f^{-1}(B(f(y),\frac{p}{r_{n}}))).$$

Then we have the following claim:

\begin{claim}
\label{medidahitting}
Given $\epsilon>0$ we have that $\mu(\SW_{p}^{f}(y,\{r_{n}\},T))<\epsilon$ if and only if for every $m\in\N$(sufficiently large) there exists $l>m$(sufficiently large) such that $\mu\left(\bigcup_{n=m}^{l}T^{-n}(f^{-1}(B(f(y),\frac{p}{r_{n}})))\right)<\epsilon$.
\end{claim}

In fact, if $\mu(\SW_{p}^{f}(y,\{r_{n}\},T))<\epsilon$ then we get that:
$$\lim_{m\to +\infty}\mu\left(\bigcup_{n\geq m}T^{-n}(f^{-1}(B(f(y),\frac{p}{r_{n}})))\right)<\epsilon$$
which implies that $\mu\left(\bigcup_{n\geq m}T^{-n}(f^{-1}(B(f(y),\frac{p}{r_{n}})))\right)<\epsilon$ for all $m\in\N$ sufficiently big. Then we have that
$\lim_{l\to +\infty}\mu\left(\bigcup_{n=m}^{l}T^{-n}(f^{-1}(B(f(y),\frac{p}{r_{n}})))\right)<\epsilon$ for all $m\in\N$. So, we have that for all $m\in\N$ there exists $l>m$(sufficiently large) such that: $$\mu\left(\bigcup_{n=m}^{l}T^{-n}(f^{-1}(B(f(y),\frac{p}{r_{n}})))\right)<\epsilon.$$
On the other hand let us suppose that for all $m\in\N$ there exists $l>m$(sufficiently large) such that: $$\mu\left(\bigcup_{n=m}^{l}T^{-n}(f^{-1}(B(f(y),\frac{p}{r_{n}})))\right)<\epsilon.$$
Then, letting $l\to +\infty$ we get that $\mu\left(\bigcup_{n\geq m}T^{-n}(f^{-1}(B(f(y),\frac{p}{r_{n}})))\right)<\epsilon$ for all $m\in\N$ and then we get that $\mu(\SW_{p}^{f}(y,\{r_{n}\},T))<\epsilon$. This proves the claim \ref{medidahitting}.
\\ \\
Now, let us define for each $\epsilon>0$ the set:
$$\SG_{\epsilon,p}=\{T\in\SM[M,\mu] : \mu(\SW_{p}^{f}(y,\{r_{n}\},T))<\epsilon\}$$
and for each $\epsilon>0$ and $l>m$:
$$\SG_{\epsilon,p}^{m,l}:=\{T\in\SM[M,\mu] : \mu\left(\bigcup_{n=m}^{l}T^{-n}(f^{-1}(B(f(y),\frac{p}{r_{n}})))\right)<\epsilon\}.$$

Then, the next claim is the following:

\begin{claim}
\label{gdelta}
Each set $\SG_{\epsilon,p}^{m,l}$ is open in the uniform topology and $\SG_{\epsilon,p}=\bigcap_{m\in\N}\bigcup_{l>m}\SG_{\epsilon,p}^{m,l}$. In particular, $\SG_{\epsilon,p}$ is a $G_{\delta}$ set. Furthermore, $\SG_{\epsilon,p}$ is dense in $\SM[M,\mu]$.
\end{claim}
The equality follows from lemma \ref{medidahitting}. The opening follows from the same argument of claim \ref{open} and will be omitted. So, let us prove the density.
Using that $\mu(f^{-1}(\{f(y)\})=0$ we get that:
$$Per(T)\subset X-W_{p}^{f}(y,\{r_{n}\},T)\,\,(mod\mu)$$
and then:
$$\{T\in\SM[M,\mu] : \mu(Per(T))>1-\epsilon \}\subset \SG_{\epsilon,p}$$
and follows from lemma \ref{Alpern} that $\SG_{\epsilon,p}$ is dense. The claim \ref{gdelta} is over now.
\\ \\
So, we have that $\SG_{\epsilon,p}$ is a generic set for every $\epsilon>0$. Now, if we define:
$$\SG_{p}:=\bigcap_{q\in\N}\SG_{1/q,p}$$
then $\SG_{p}$ is a generic set and if $T\in\SG_{p}$ then $\mu(\SW_{p}^{f}(y,\{r_{n}\},T)=0$, which shows that $\liminf_{n\to +\infty}r_{n}d(f(T^{n}(x),f(y)))\geq p$ for $\mu$ a.e. $x\in M$ if $T\in\SG_{p}$. Now, if we define:
$$\SG:=\bigcap_{p\in\N}\SG_{p}$$
then $\SG$ is a generic set and $\liminf_{n\to +\infty}r_{n}d(f(T^{n}(x)),f(y))=+\infty$ for $\mu$ a.e. $x\in M$ if $T\in\SG$. This finish the proof of theorem \ref{teoprincipalht}.
  
\end{proof}

\begin{proof}(Theorem \ref{teoprincipalrtinfinite})
Let us begin with the following lemma:

\begin{lemma}
\label{Periodicoinfinito}
Let be $\lambda$ the lebesgue measure on $\R^{n}$. Then, for every $\epsilon>0$ the set of homeomorphisms $T\in \SM(\R^{n},\lambda)$ such that $\lambda((Per(T))^{c})<\epsilon$ is dense in $\SM[\R^{n},\lambda]$ with respect to compact open topology, where $Per(T)$ denotes the set of periodic points of $T$.
\end{lemma}

\begin{proof}
Let $\SC(f,K,\delta)$ a basic open set with respect to compact open topology. We can suppose without loss of generality that $K$ is a compact cube. Given $C$ a compact cube containing $K\cup f(K)$ in its interior. Then it follows from Lemma 12.2 of \cite{BF} that there exists $\hat{f}\in\SM[\R^{n},\lambda]$ which leaves $C$ invariant, and agrees with $f$ on $K$.
Now, it follows from the proof of lemma \ref{Alpern} that there exists $g\in\SM(C,\lambda)$, where $\SM(C,\lambda)$ denotes the set of homeomorphisms of $C$ such that the lebesgue measure is invariant, such that $\lambda(Per(g))>\lambda(C)-\epsilon$ and $d(g(x),\hat{f}(x))<\delta$ for $\lambda$-qtp $x\in C$. Let $C_{1}\supset C$ be a cube concentric to $C$ such that $\lambda(C_{1}-C)<\epsilon$ and extend $g$ to a homeomorphism of $C_{1}$ onto itself such that $g$ it is equal to identity on the boundary of $C_{1}$.
Let $A:=C_{1}-C$ and define for each borel set $B\subset A$ the measures:
$$\mu_{1}(B)=\lambda(B)\,\,\,and\,\,\,\mu_{2}(B)=\lambda(g(B)).$$
Therefore by the Homeomorphic Measures Theorem (see Corollary A2.6 on \cite{BF}) there exists a homeomorphism $h:A\to A$ such that $\mu_{2}(h(B))=\mu_{1}(B)$ for each borel set $B\subset A$ and such that $h$ is the identity on the boundary of $A$. Now, define $k:C_{1}\to C_{1}$ such that:
$$k(x)=g(x)\,\,\,in\,\,C\,\,\,and\,\,\,k(x)=g(h(x))\,\,\,in\,\,C_{1}-C$$
and define $\hat{g}:\R^{n}\to \R^{n}$  such that:
$$\hat{g}(x)=k(x)\,\,\,in\,\,C_{1}\,\,\,and\,\,\,\hat{g}(x)=Id\,\,\,in\,\,C_{1}^{c}.$$
Note that $\hat{g}\in\SM(\R^{n},\lambda)$, $\lambda(Per(\hat{g})^{c})<2\epsilon$ and $\hat{g}\in\SC(f,K,\delta)$ which shows the density. 

To complete the proof of Theorem \ref{teoprincipalrtinfinite} we follow the same ideas of Theorem \ref{teoprincipalrt}. In fact let us define for each $T\in\SM[\R^{n},\lambda]$ the set:
$$\SH(T)=\{x\in \R^{n} : \liminf_{n\to \infty} r_{n}|T^{n}(x)-x|<+\infty\}.$$
Given $n,k\in\N$ we define:
$$Y_{n,k}(T)=\{x\in \R^{n} : r_{n}|T^{n}(x)-x|<k\}.$$
The following two claims holds in the same way that the compact case:
\begin{claim}
\label{equalitynoncompact}
$$\SH(T)=\bigcup_{k>0}\bigcap_{m\geq 1}\bigcup_{n\geq m}Y_{n,k}(T)$$
\end{claim}

\begin{claim}
\label{measurenoncompact}
$\lambda([\SH(T)]^{c})<\epsilon$ if and only if there exists $k>0$ such that for all $m\in \N$ , there exists a positive integer $l>m$ such that $\lambda(\bigcap_{n=m}^{l}[Y_{n,k}(T)]^{c})<\epsilon$.
\end{claim}

Now, let us define for each $\epsilon>0$ and $l>m$ in $\N$ the set:
$$\SH_{\epsilon}^{m,l}:=\{T\in \SM[\R^{n},\lambda]:there\,\,\,exists\,\,\,k>0\,\,\, such\,\,that\,\, \lambda(\bigcap_{n=m}^{l}(Y_{n,k}(T))^{c})<\epsilon\}$$

Then we have:

\begin{claim}
\label{opennoncompact}
$\SH_{\epsilon}^{m,l}$ is open in the compact-open topology
\end{claim}
\begin{proof}(claim\ref{opennoncompact})
Let $T\in \SH_{\epsilon}^{m,l}$. Then we have to prove that if $S$ is sufficiently close of $T$ then $S\in\SH_{\epsilon}^{m,l}$.
Note that it is enough to prove that if $S$ is sufficiently close to $T$ then $Y_{n,k}(T)\subset Y_{n,k}(S)$. By triangle inequality:
$$r_{n}|S^{n}(x)-x|\leq r_{n}|S^{n}(x)-T^{n}(x)|+r_{n}|T^{n}(x)-x|$$
Using this inequality and the metric $U$ of the compact open topology we get that if $S$ is close of $T$ then $Y_{n,k}(T)\subset Y_{n,k}(S)$ and this prove the opening.
\end{proof}

Now, let us define for each $\epsilon>0$ the set
$$\SH_{\epsilon}=\{T\in \SM[\R^{n},\lambda] : \mu((\SH(T))^{c})<\epsilon \}$$
and using the claim \ref{measurenoncompact} we get:
$$\SH_{\epsilon}=\bigcap_{m\in\N}\bigcup_{l>m}\SH_{\epsilon}^{m,l}$$
and then $\SH_{\epsilon}$ is a $G_{\delta}$ set and dense by lemma \ref{Periodicoinfinito}. The residual set is then given by
$$\SH:=\bigcap_{n\in\N}\SH_{1/n}$$
which proves the theorem \ref{teoprincipalrtinfinite}.

\end{proof}

\end{proof}

\begin{proof}(Corollary \ref{genericdecay})
Let $T\in\SM[M,\mu]$ be a homeomorphism with superpolynomial decay of correlations.
Then, if we take $\beta>d_{\mu}(y)$ and $t_{n}:=n^{\frac{-1}{\beta}}$ we have that $0<\frac{1}{\beta}<\frac{1}{d_{\mu}(y)}$ and then it follows from the proof of main theorem of \cite{Ga} that the sequence of balls $\{B(y,t_{n})\}$ is Borel-Cantelli, which means that
$$\mu(\{x\in M:d(T^{n}(x),y)<t_{n}\,\,for\,\,infinitely\,\,many\,\,n\})=1.$$
This implies that 
$$\liminf_{n\to +\infty}n^{\frac{1}{\beta}}d(T^{n}(x),y)\leq 1$$
for $\mu$-a.e.$x\in M$. On the other hand, using theorem \ref{teoprincipalht} with $f=Id$ and $r_{n}=n^{\frac{1}{\beta}}$ we get a residual set $\SG_{1}\subset\SM[M,\mu]$ such that $T\in\SG_{1}$ implies:
$$\liminf_{n\to +\infty} n^{\frac{1}{ \beta}}d(T^{n}(x),y)=+\infty\,\,\,for\,\,\mu\,\,a.e.\,\,x\in M.$$
This shows that in this generic set the decay of correlations is not super-polynomial and the proof is over.

\end{proof}

%%%%%%%%%%%%%%%%%%%%%%%%%%%%%%%%%%%%%%%%%%%%%%%%%%%%%%%%%%%%%%%%%%%%%%%%%%%
\section*{Acknowledgments} The author would like to thank Alexander Arbieto and Jairo Bochi for helpful discussions about the paper. The author is grateful to IMPA for the nice environment provided during the beginning of this paper. Work partially supported by CAPES from Brazil.

%%%%%%%%%%%%%%%%%%%%%%%%%%%%%%%%%%%%%%%%%%%%%%%%%%%%%%%%%%%%%%%%%%%%%%%%%%
\bibliographystyle{alpha}

\begin{thebibliography}{10}


\bibitem[Aa]{Aa}
J. Aaronson,
\newblock An introduction to Infinite Ergodic Theory
\newblock American Mathematical Society 1991.


\bibitem[BF]{BF}
B.Bollobas, W.Fulton, A.Katok, F.Kirwan and P.Sarnak,
\newblock Typical Dynamics of Volume Preserving Homeomorphisms.
\newblock Cambridge University Press 2000.

\bibitem[Bon]{Bon}
C.Bonanno, S.Galatolo and S. Isola,
\newblock Recurrence and Algorithmic Information.
\newblock {\em Nonlinearity}, {\bf 17} (2004),no. 3, 1057-1074.

\bibitem[Bo]{Bo}
M. Boshernitam,
\newblock Quantitative Recurrence Results.
\newblock {\em Invent. Math.}, {\bf 113} (1993), 617-631.


\bibitem[DF]{DF}
F. Daalderop et R. Fokkink ,
\newblock Chaotic homeomorphisms are generic. 
\newblock {\em Topology Appl. }, {\bf 102} (2000), no. 3, 297-302.


\bibitem[Ga]{Ga}
S. Galatolo ,
\newblock Dimension via waiting time and recurrence.
\newblock {\em Math. Res. Lett.}, {\bf 12} (2005), 377-386.




\bibitem[Mu]{Mu}
James R. Munkres,
\newblock Topology a First Course.
\newblock Prentice-Hall, Inc., Englewood Cliffs, New Jersey.




\bibitem[W]{W}
P. Walters, 
\newblock An introduction to ergodic theory.
\newblock Graduate Texts in Mathematics, 79. Springer-Verlag, New York-Berlin, 1982.







\end{thebibliography}

\end{document}